\documentclass{amsart}
\usepackage{amsmath, amssymb, amsthm, verbatim, hyperref, epsfig, graphicx, epstopdf}
\usepackage[all, cmtip]{xy} 
\usepackage{ulem}
\usepackage{hyperref}
\newcounter{counter}
\numberwithin{counter}{section}

\newcommand{\arxiv}[1]{\href{http://arxiv.org/abs/#1}{\tt arXiv:\nolinkurl{#1}}}
\newcommand{\doi}[1]{\href{http://dx.doi.org/#1}{{\tt DOI:#1}}}

\newcommand{\googlebooks}[1]{(preview at
\href{http://books.google.com/books?id=#1}{google books})}

\newtheorem{thm}[counter]{Theorem}
\newtheorem{mydef}[counter]{Definition}
\newtheorem{lma}[counter]{Lemma}
\newtheorem{cor}[counter]{Corollary}
\newtheorem{prop}[counter]{Proposition}

\theoremstyle{remark}
\newtheorem{rem}[counter]{Remark}
\theoremstyle{theorem}
\newcommand{\Inv}{\uline{\operatorname{Inv}}}
\newcommand{\Eq}{\uline{\operatorname{Eq}}}
\newcommand{\Out}{\uuline{\operatorname{Out}}}
\newcommand{\BrPic}{\uuline{\operatorname{BrPic}}}
\newcommand{\cC}{\mathcal{C}}
\newcommand{\cI}{\mathcal{I}}
\newcommand{\cF}{\mathcal{F}}
\newcommand{\CC}{\mathbb{C}}
\newcommand{\cD}{\mathcal{D}}
\newcommand{\cN}{\mathcal{N}}
\newcommand{\ZZ}{\mathbb{Z}}
\newcommand{\cM}{\mathcal{M}}
\newcommand{\ot}{\otimes}
\newcommand{\bt}{\boxtimes}
\newcommand{\id}{\operatorname{id}}
\newcommand{\Hom}{\operatorname{Hom}}

\def\hpic #1 #2 {\mbox{$\begin{array}[c]{l} \epsfig{file=#1,height=#2}
\end{array}$}}
\title[Cyclic extensions of fusion categories]{Cyclic extensions of fusion
categories via the Brauer-Picard groupoid}
\address[p.grossman@unsw.edu.au]{Pinhas~Grossman, School of Mathematics and Statistics, University of New South Wales, SYDNEY NSW 2052 AUSTRALIA
}%

\address[djordan@math.utexas.edu]{David~Jordan, Dept. of Mathematics, The University of Texas at Austin, 1 University Station C1200, Austin, TX 78712-0257 
}%

\author{Pinhas Grossman, David Jordan, Noah Snyder}
\address[nsnyder@math.indiana.edu]{Noah Snyder, Dept. of Mathematics, Indiana University, Rawles Hall, 831 E 3rd St, Bloomington, IN 47405
}%

\date{}
\begin{document}

\begin{abstract}
We construct a long exact sequence computing the obstruction space,
$\pi_1\BrPic(\cC_0)$, to $G$-graded extensions of a fusion category $\cC_0$. 
The other terms in the sequence can be computed directly from the fusion
ring of $\cC_0$.  We apply our result to several examples coming from
small index subfactors, thereby constructing several new fusion categories as
$G$-extensions.  The most striking of these is a
$\mathbb{Z}/2\mathbb{Z}$-extension of one of the Asaeda-Haagerup fusion
categories, which is one of only two known $3$-supertransitive fusion categories outside the ADE
series.

In another direction, we show that our long exact sequence appears in
exactly the way one expects: it is part of a long exact sequence of homotopy groups associated to a
naturally occuring fibration.  This motivates our constructions, and gives another example
of the increasing interplay between fusion categories and algebraic topology.
\end{abstract}

\maketitle
\section{Introduction}
In this paper we construct several new fusion categories related to the
Asaeda-Haagerup subfactor \cite{MR1686551} and the related $AH+1$ and $AH+2$
subfactors \cite{MR2812458,1202.4396}.  These constructions require the
calculation of certain obstruction groups which appear in the theory of
$G$-extensions of fusion categories, and we construct a long exact sequence
which allows us to do this calculation.  This long exact sequence is in turn a
consequence of a certain homotopy fibration of higher groups, as we explain.

Let us begin by specifying more precisely the concrete problem we wish to solve.
 Consider a finite index, finite depth subfactor pair $N\subset M$, which is
{\it self-dual}, i.e. is equipped with an equivalence of fusion categories
between the {\it principal even part} $\cC_0=\langle_NM_N\rangle \subseteq
N\text{-mod-}N$ and its dual $(\cC_0)_{\cC_1}^*$ with respect to the
$\cC_0$-module category $\cC_1=\langle _NM_M\rangle\subseteq N\text{-mod-}M$. 
Examples include the Izumi-Xu and $AH+2$ subfactors.   In this situation $\cC_1$
is a $\cC_0$-bimodule category, and it is natural to ask whether we can combine
the common even part and the odd part to form a new fusion category $\cC=\cC_0
\oplus \cC_1$.

The theory of $G$-graded extensions of a fusion category, introduced in
\cite{MR2183279} and developed in \cite{MR2677836}, provides answers to
precisely such questions. A fusion category $\cC$ is graded by $G$ if we have a
decomposition $\cC=\oplus_{g\in G} \cC_g$, compatible with tensor product.  Thus
the categories considered above are instances of $\ZZ/2\ZZ$-graded extensions of
the even part $\cC_0$.

The paper \cite{MR2677836} constructs an equivalence between $G$-graded
extensions of a fixed category $\cC_0$, and homomorphisms $\rho:G\to
\BrPic(\cC_0)$ to the categorical $2$-group $\BrPic(\cC_0)$ of $\cC_0$-bimodule
categories.  Standard arguments in algebraic topology then reduce the existence
of such extensions to vanishing of obstruction classes $o_k\in
H^{k}(G,\pi_{k-2}\BrPic(\cC_0))$, for $k=3,4$.  When $G=\ZZ/2\ZZ$, or more
generally when $G$ is cyclic, $o_4$ is automatically trivial, so that we need
only contend with $o_3 \in H^3(G,\pi_1\BrPic(\cC_0))$.

The main technical tool we develop in this paper is an exact sequence,
\begin{equation} \label{eqn:SES}\ast \to \Hom(U(\cC_0),\CC^\times) \to
\pi_1\BrPic(\cC_0) \to \operatorname{Inv}(\cC_0),\end{equation}
which realizes $\pi_1\BrPic(\cC_0)$ as an extension of a subgroup of the group
$\operatorname{Inv}(\cC_0)$ of invertible objects of $\cC_0$, by the group
$\Hom(U(\cC_0),\CC^\times)$ of characters of the universal grading group of
$\cC_0$ .  The virtue of \eqref{eqn:SES} is the other two groups can be read off
directly from the fusion data of $\cC_0$; we leverage this to show that
$H^3(\ZZ/2\ZZ,\pi_1\BrPic(\cC_0))$ is trivial in our examples, so that the
obstruction $o_3$ vanishes automatically.

In Section \ref{sec:MainResults}, we give an {\it ad hoc} construction of the
sequence \eqref{eqn:SES}, which relies on an identification
$\pi_1\BrPic(\cC_0)\cong \operatorname{Inv}(Z(\cC_0))$, from \cite{MR2677836}.  However,
\eqref{eqn:SES} may be more properly understood a consequence of the following:

\newtheorem*{homotopy-thm}{Theorem \ref{homotopy-thm}}
\begin{homotopy-thm}We have a homotopy fiber sequence,
\[\Inv(\cC_0)\xrightarrow{F_-}\Eq(\cC_0)\xrightarrow{M_-} \Out(\cC_0),\]
\end{homotopy-thm}
Here $\Inv$ and $\Eq$ denote categorical $1$-groups of invertible objects and
tensor automorphisms, respectively, and $\Out$ is a certain full subgroup of
$\BrPic$ with the same $\pi_1$ and $\pi_2$.  In Corollary \ref{app-LES}, we
deduce the sequence \eqref{eqn:SES} as a fragment of the long exact sequence in
homotopy groups induced by Theorem \ref{homotopy-thm}.  The proof of Theorem
\ref{homotopy-thm} is delayed until Section \ref{sec4}.

In Section \ref{sec:AH}, we turn to applications of the sequence \eqref{eqn:SES}
to subfactors.  Our primary application is the construction of a new fusion
category $\mathcal{AH}+2$, built from the even and odd parts of the self-dual
subfactor $AH+2$.  This new fusion category $\mathcal{AH}+2$ is particularly
notable because it is, along with Morrison-Penneys's 4442 fusion category \cite{1208.3637}, one of the first $3$-supertransitive fusion category
outside of the ADE families.  It is generated by an object of dimension
$\frac{1+\sqrt{17}}{2}$.

In fact, $\mathcal{AH}+2$ is just one of eighteen new examples of fusion
categories we build as $\ZZ/2\ZZ$-extensions of Asaeda-Haagerup type
categotries.  By \cite{1202.4396}, the principal even parts $\mathcal{AH}_1$,
$\mathcal{AH}_2$, $\mathcal{AH}_3$ of $AH, AH+1, AH+2$, respectively, each have
three non-trivial bimodule categories up to equivalence.  We have:

\newtheorem*{thm:AHconst}{Theorem \ref{thm:AHconst}}
\begin{thm:AHconst}
Each of the three non-trivial bimodule categories over each $\mathcal{AH}_i$,
$i=1,2,3$, is the odd component of exactly two $\mathbb{Z}/2\mathbb{Z}$-graded
extensions of $\mathcal{AH}_i$ . 
\end{thm:AHconst}

These techniques work well more generally when applied to any fusion category
coming from a $2$-supertransitive subfactor.  In Section \ref{sec:Izumi-Xu}, we
give one other source of such examples, the near group categories $\cC_p$
associated to $\ZZ/p\ZZ$.  These are fusion categories of the form
$Vec(\ZZ/p\ZZ) \oplus Vec$, which are not necessarily $\ZZ/2\ZZ$-graded.  The
group algebra $\CC[\ZZ/p\ZZ]$ is an algebra in $\cC_p$, and we we let $\cM_p$
denote its category of modules.  We say that $\cC_p$ is self-dual if
$\cC_p\cong(\cC_p)_{\cM_p}^*$. We have:

\newtheorem*{thm:Cp}{Theorem \ref{thm:Cp}}
\begin{thm:Cp} 
Suppose that $p>2$, that $\cC_p$ is a self-dual $\mathbb{Z}/p\mathbb{Z}$
near-group category with trivial outer automorphism group.  Then there exist
exactly two $\mathbb{Z}/2\mathbb{Z}$-extensions of $\cC_p$ by $\cM_p$.
\end{thm:Cp}

Applying Theorem \ref{thm:Cp} and Han's thesis \cite{han-2221}, we give a third
proof of a result first proved by Ostrik \cite[Appendix]{MR2786219} and second
by Morrison-Penneys \cite{1208.3637}, establishing the existence of a certain
$\ZZ/2\ZZ$-graded extension of the Izumi-Xu category $\mathcal{IX}$.  It is our
hope that Theorem \ref{thm:Cp} will find application in more near-group
examples; however this will require developing techniques to establish
self-duality and to calculate outer automorphism groups.

\subsection{Acknowledments}
It is our pleasure to thank Andrew Blumberg, Pavel Etingof, and Aaron Royer for helpful
conversations about homotopy theoretic techniques.  This collaboration began at
the 2011 Subfactors in Maui conference supported by DARPA grant 
HR0011-12-1-0009. Pinhas Grossman was supported by a fellowship at IMPA and by
NSF grant DMS-0801235.  David Jordan was supported by NSF postdoctoral
fellowship 1103778. Noah Snyder was supported by an NSF postdoctoral fellowship,
DARPA grant HR0011-12-1-0009, and as a visiting scientist at MPIM.

\section{Preliminaries}
\subsection{Categorical $n$-groups}
In order to state the results of  \cite{MR2677836} on extensions of fusion
categories we will need to use the notions of higher groupoids and higher
categorical groups.  We will need only the notion of a categorical $0$-group
(i.e. a group), $1$-group, and $2$-group in this paper.  

By an $n$-groupoid, we will mean an $n$-category $\cC$, all of whose morphisms
at all levels are invertible.  Recall from \cite{MR2677836} that a categorical
$n$-group $\mathcal{G}$ is a monoidal $n$-groupoid, in which all objects are
invertible, or equivalently, an $(n+1)$-groupoid with a single object. A
homomorphism of categorical $n$-groups is a monoidal functor of $n$-groupoids or
a functor of connected $(n+1)$-groupoids in that formulation.

For $m>k$, we can regard any categorical $k$-group as a categorical $m$-group,
with trivial morphisms in degree $k+1,\ldots, m$.  Thus we will often speak of a
homorphism from, say, a categorical $0$-group to a categorical $2$-group, and
that will mean a homomorphism regarding both as categorical $2$-groups.

\begin{rem}
The classifying space construction $\cC\mapsto B\cC$ defines an equivalence
between the category of $n$-groupoids and homotopy $n$-types; for this reason
(more precisely, invertibility of morphisms at all levels), the well-known
subtleties in the foundations of higher categories are largely absent from the
theory of higher groupoids and higher categorical groups.  In particular, a
categorical $n$-group may be regarded as a connected homotopy $n+1$-type, just
as a group may be identified with the connected homotopy $1$-type of its
classifying space.  This identification shifts dimensions: we have $\pi_k\cC =
\pi_{k+1}B\cC$, canonically.  
\end{rem}

\begin{rem}
Categorical $n$-groups are often called $(n+1)$-groups in the literature (e.g. \cite{MR2068521}).  Both
indexings are reasonable, depending on whether you think of a group as a set
with an operation or as a $1$-category with only one object.
\end{rem}

\subsection{Fusion categories and their extensions}
In this subsection we recall the extension theory of fusion categories developed
in \cite{MR2677836}.

\begin{mydef}\cite{MR2183279}
A fusion category over $\CC$ is a finite $\CC$-linear semisimple rigid monoidal
category with simple identity object. 
\end{mydef}

For definitions of module categories, bimodule categories, tensor products of
bimodule categories, and invertibility, see \cite{MR2677836}; in this paper we
assume all module categories are semisimple.

\begin{mydef} \cite{MR2677836}\label{def:brpic}
The Brauer-Picard groupoid of a fusion category $\mathcal{C}$ is a $3$-groupoid
whose:
\begin{itemize}

\item objects are fusion categories which are Morita equivalent to $\mathcal{C}$

\item $1$-morphisms are invertible bimodule categories between such fusion
categories

\item $2$-morphisms are equivalences of such bimodule categories

\item $3$-morphisms are isomorphisms of such equivalences.
\end{itemize}

The Brauer-Picard categorical $2$-group $\BrPic(\cC)$ is the full subgroupoid of
the Brauer-Picard groupoid whose only object is  $\mathcal{C}$. 

\end{mydef}

An extension of a fusion category $\mathcal{C} $ by a finite group $G$ is a
$G$-graded fusion category whose $0$-graded part is equivalent to $\mathcal{C}
$. There is a natural notion of equivalence of $G$-extensions.

\begin{thm} \cite{MR2677836} \label{ENOthm}
(a) Equivalence classes of $G$-extensions of $\mathcal{C}$ are given by
categorical $2$-group homomorphisms from $G$ to $\BrPic(\cC)$.\\
(b) Such homomorphisms (and hence, $G$-extensions) are parameterized by triples
$(c,M,\alpha)$, where $c$ is a group homomorphism $c:G\to \operatorname{BrPic}(\cC)$, $M$
belongs to a certain $H^2(G,\pi_1\BrPic(\cC))$-torsor $T^2_c$, and $\alpha$
belongs to a certain $H^3(G,\pi_1\BrPic(\cC))$-torsor $T^3_{c,M}$.\\
(c) Certain obstruction classes $o_3(c)\in H^3(G,\pi_1\BrPic(\cC))$ and
$o_4(c,M)\in H^4(G,\pi_2\BrPic(\cC))$ must vanish for $(c,M,\alpha)$ to
determine an extension.
\end{thm}

\subsection{Subfactors}

A \textit{subfactor} is a unital inclusion $N \subseteq M$ of II$_1$ factors. A
subfactor $N \subseteq M$ has finite index if $M$ is a finitely-generated
projective module over $N$ \cite{MR696688, MR860811}. In this case, the $N-N$
bimodule ${}_N M {}_N $ tensor generates a semisimple unitary rigid monoidal
category of $N-N$ bimodules, called the \textit{principal even part} of the
subfactor. The $N-M$ bimodule ${}_N M {}_M $ generates a module category over
the principal even part; the dual category of this module category, which is the category 
of $M-M$ bimodules tensor generated by ${}_M M {}_N \otimes_N {}_N M {}_M$, is called the
\textit{dual even part}. The subfactor is said to have \textit{finite depth} if
the even parts are fusion categories, i.e. if they each have finitely many
simple objects, up to isomorphism. 

If a subfactor has the same principal and dual principal parts, it is natural to
ask whether there is a $\mathbb{Z}/2\mathbb{Z}$-extension whose $0$-graded part
is the even part of the subfactor, and whose $1$-graded part is the odd part of
this subfactor.  In particularly nice situations (where the generator of the odd
part $\cC_1$ becomes self-dual in $\cC$) this can be understood directly in
terms of subfactors or planar algebras.  In terms of the factors, you get such
an extension when you can realize $N \subset M$ as coming from a self-dual
bimodule over $N$.  In terms of planar algebras, such an extension tells you
that the shaded planar algebra comes from an unshaded planar algebra.

Note that although we are studying examples coming from subfactors, we do not
explicitly address unitarity of the extension here.  In particular, the
$\mathbb{Z}/2\mathbb{Z}$-extensions come in pairs, and it does not seem
reasonable to expect that they would both be unitary.

\section{Computing $\pi_1$ of the Brauer-Picard group} \label{sec:MainResults}
The group $\pi_1\BrPic(\cC) \cong \operatorname{Inv}(Z(\cC))$ houses the primary obstruction in
the extension theory of $\cC$. In this section, we use techniques from
elementary homotopy theory to compute this group.  In addition to the
categorical $2$-group $\BrPic(\cC)$, the main examples we consider are as
follows.

\begin{mydef}
The categorical $2$-group $\Out(\cC)$, of outer equivalences of $\cC$, is the
full $2$-subgroup of invertible bi-module categories that are equivalent to
$\cC$ as left $\cC$-module categories.
\end{mydef}

\begin{mydef}
The categorical $1$-group $\Inv(\cC)$ is the subcategory of invertible objects
in $\cC$ and their isomorphisms:
\begin{itemize}
\item Objects are the invertible objects $(g,h, \ldots)$ in $\cC$,
\item $\Hom(g,h)$ are the isomorphisms ($\phi,\psi$,\ldots) of such.
\end{itemize}
\end{mydef}

\begin{mydef}
The categorical $1$-group $\Eq(\cC)$ is the category of tensor auto-equivalences
of $\cC$:
\begin{itemize}
\item Objects are tensor auto-equivalences ($F, G, \ldots$) of $\cC$,
\item $\Hom(F,G)$ are monoidal natural isomorphisms of such.
\end{itemize}
\end{mydef}

The assignments $F_-:\Inv(\cC)\to\Eq(\cC)$, sending $g$ to the functor,
$$F_g(X):=g\ot X\ot g^*,$$
and $M_-:\Eq(\cC)\to\Out(\cC)$, sending $F$ to the outer bimodule category
$M_F:=\cC$, with tensor product defined, for $X,Y\in \cC$, and $m\in M_F$, by:
$$X\ot_{M_F} m \ot_{M_F} Y := X\ot m\ot F(Y).$$
can each be upgraded to functors of categorical $2$-groups (see \S \ref{sec4}).
Our main result follows; its proof is deferred until the final section.

\begin{thm} We have a homotopy fiber sequence,\label{homotopy-thm}
$$\Inv(\cC)\xrightarrow{F_-}\Eq(\cC)\xrightarrow{M_-} \Out(\cC).$$
\end{thm}

\begin{cor} \label{maincor}
We have the long exact sequence of homotopy groups,
$$\xymatrix{
&&0\ar[r]&\pi_2\Out(\cC)\ar `r/8pt[d] `/10pt[l] `^dl[ll]|{} `^r/3pt[dll]
[dll]&\\
& \pi_1\Inv(\cC) \ar[r]& \pi_1\Eq(\cC) \ar[r] &\pi_1\Out(\cC) \ar`r/8pt[d]
`/10pt[l] `^dl[ll]|{} `^r/3pt[dll] [dll]&\\ 
& \pi_0\Inv(\cC) \ar[r]& \pi_0\Eq(\cC) \ar[r] & \pi_0\Out(\cC)\ar[r]& 0.}$$
\end{cor}


\begin{prop}
We have the following isomorphisms:
\begin{itemize}
\item $\pi_2\Out(\cC)\cong \pi_1\Inv(\cC)\cong \CC^\times$.
\item $\pi_1\Eq(\cC)\cong Aut(\id_\cC)\cong \Hom(U(\cC),\CC^\times)$.
\item $\pi_1\Out(\cC)\cong \pi_1(\BrPic(\cC)) \cong \operatorname{Inv}(Z(\cC)),$ the group of
iso-classes of invertible objects in $Z(\cC)$.
\item $\pi_0\Inv(\cC) \cong \operatorname{Inv}(\cC),$ the group of iso-classes of invertible
objects in $\cC$.
\item $\pi_0\Eq(\cC) \cong \operatorname{Eq}(\cC),$ the group of tensor auto-equivalences of
$\cC$.
\item $\pi_0\Out(\cC) \cong \operatorname{Out}(\cC)$, the group of tensor auto-equivalences,
modulo inner equivalences.
\end{itemize}
\end{prop}

Combining these identities with Theorem \ref{homotopy-thm}, and computing the
connecting homomorphisms, yields:

\begin{cor}\label{app-LES}
We have a long exact sequence:
$$\ast \to \Hom(U(\cC),\CC^\times) \to \pi_1\BrPic(\cC) \xrightarrow{\delta}
\operatorname{Inv}(\cC) \xrightarrow{F_-} \operatorname{Eq}(\cC) \xrightarrow{M_-} \operatorname{Out}(\cC) \to \ast.$$
\end{cor}

The corollary allows us to compute $\pi_1\BrPic(\cC) \cong \operatorname{Inv}(Z(\cC))$ in terms
of its neighbors $\operatorname{Inv}(\cC)$ and $\Hom(U(\cC),\CC^\times)$, which may be read
from the fusion rules of $\cC$, and do not require us to comprehend the entire
center $Z(\cC)$.

Since Corollary \ref{app-LES} is what we will use in applications, we give an
independent, and elementary proof of it below.  While Theorem \ref{homotopy-thm}
gives a more conceptual explanation, it is not strictly necessary for
applications, and so we postpone its proof, and the consequent derivation of Corollary
\ref{app-LES}.

\subsection{Independent proof of Corollary \ref{app-LES}}
The homomorphism $M_-$ assigns to a tensor automorphism $\phi$ of $\cC$ the
bimodule $M_\phi=\cC$, with regular left action, and with right action twisted
by $\phi$ (see Section \ref{sec4} for a precise definition).
The homomorphism $F_-$ assigns to an invertible object $g\in\cC$ the tensor
automorphism $F_g: X\mapsto g\otimes X\otimes g^*$.  The homomorphism $\delta$
is induced by the forgetful functor $Z(\cC)\to \cC$.

Exactness at $\operatorname{Out}(\cC)$ and $\operatorname{Eq}(\cC)$ are the facts that all outer bimodules
come from equivalences and that such a bimodule is trivial if, and only if, the
equivalence is inner \cite[\S4.3]{MR2677836}.  Exactness at $\operatorname{Inv}(\cC)$ is the
fact that an isomorphism $F_g\cong\id_\cC$ yields a half-braiding on $g$.

Thus it remains only to identify $\Hom(U(\cC),\CC^\times))$ with the kernel of
$\delta$.  Thus, we consider half-braidings $\sigma_{\mathbf {1},-}$ of the
tensor unit.  Since we have canonical isomorphisms $\mathbf{1}\ot X\cong X \cong
X\ot \mathbf{1}$, the data of such a half-braiding is a scalar
$c_X\in\CC^\times$ for each simple object $X\in \cC$.  The Yang-Baxter equation
and naturality condition for $\sigma_{\mathbf{1},-}$ imply that, for every
simple object $Z$ in the decomposition of $X\ot Y$, we have $c_Z=c_Xc_Y$.  Thus
the level sets of $c$ determine a grading of $\cC$ by a subgroup of
$\CC^\times$, yielding the required element of $\Hom(U(\cC),\CC^\times)$ via
the universal property of $U(\cC)$.  Finally, only a trivial such homomorphism
can give rise to the trivial half-braiding.

\section{Applications to subfactors} \label{sec:applications}
The results of the last section, combined with the obstruction theory of
\cite{MR2677836}, allow us to construct several new fusion categories, starting
with a fusion category and its invertible bimodule category.

A finite-index subfactor $N \subset M$ is called \textit{$2$-supertransitive} if
$M \cong \mathbf{1} + \mathbf{X}$ as $N-N$ bimodules, where $\mathbf{1}$ is the
trivial $N-N$ bimodule ($N$ itself) and $\mathbf{X}$ is a simple object in the
principal even part of $N \subset M$.

\begin{lma}
Let $N \subset M $ be a $2$-supertransitive finite depth subfactor. Then the
grading group of each of the even parts of $N \subset M$ is trivial.
\end{lma}

\begin{proof}
Let $\mathbf{A}=\mathbf{1}+\mathbf{X}$ be the algebra object in the principal
even part $\cN$ of $ N \subset M$.  Since the subfactor is $2$-supertransitive,
$\mathbf{X}$ is simple.  Since $\mathbf{1}+\mathbf{X}$ is an algebra, we have
that $\mathbf{X} \cong \mathbf{X}^*$ and that $\mathbf{X} \subset \mathbf{X}
\otimes \mathbf{X} \cong \mathbf{X} \otimes \mathbf{X}^*$.  Thus $\mathbf{X}$
lies in the $0$-graded part of the principal even part $\cN $ for any grading of
$\cN $; since $\mathbf{X}$ tensor generates the principal even part this means
that the even principal part has no non-trivial gradings.  Applying the same
argument to the dual subfactor yields that the dual principal part has trivial
grading group as well.
\end{proof}

\begin{cor} \label{noinvs}

Let $N \subset M $ be a $2$-supertransitive finite depth  subfactor such that
the principal even part $\cN $ has no  invertible objects except for
$\mathbf{1}$. Then $\pi_1(\BrPic(\cC)) \cong \operatorname{Inv}(Z(\cC))$ is trivial.

\end{cor}

\subsection{The Asaeda-Haagerup categories}\label{sec:AH}

We now turn to our main application, constructing several fusion categories
which are extensions of fusion categories related to the Asaeda-Haagerup
subfactor \cite{MR1686551}. 

The Asaeda-Haagerup subfactor, which we will call AH, is a finite depth
subfactor with index $\frac{5+\sqrt{17}}{2}  $. The subfactors AH+1 and AH+2,
constructed in \cite{MR2812458,1202.4396}, have indices $\frac{7+\sqrt{17}}{2} $
and $\frac{9+\sqrt{17}}{2} $, respectively. The three subfactors AH, AH+1, and
AH+2 all have the same dual even part but the principal even parts are three
distinct fusion categories, which we call $\mathcal{AH}_1 $ , $\mathcal{AH}_2$,
and $\mathcal{AH}_3$, respectively. 

The subfactors AH, AH+1, and AH+2 are all $2$-supertransitive. The categories
$\mathcal{AH}_2$ and $\mathcal{AH}_3$ have non-trivial invertible objects but
$\mathcal{AH}_1$ does not. Therefore by Corollary \ref{noinvs} we have that
$\operatorname{Inv}(Z(\mathcal{AH}_i ) )$ is trivial for $i=1,2,3 $ (since the Drinfeld center
is a Morita invariant \cite[Cor. 2.1]{MR1976233}).

We recall the following result from \cite{1202.4396}:

\begin{thm}
The Brauer-Picard group of each of the Asaeda-Haagerup fusion categories is
$\mathbb{Z}/2\mathbb{Z} \oplus \mathbb{Z}/2\mathbb{Z} $.

\end{thm}

Therefore we have three order $2$ invertible bimodule categories over each
$\mathcal{AH}_i $, for $i=1,2,3 $. Full fusion rules for each of these bimodule
categories were given in the ArXiv data supplement to \cite{1202.4396}. We now show
that each of these $9$ bimodule categories admits two fusion category
structures.

\begin{thm}\label{thm:AHconst}
Each of the three non-trivial bimodule categories over each $\mathcal{AH}_i$,
$i=1,2,3$, is the odd component of exactly two $\mathbb{Z}/2\mathbb{Z}$-graded
extensions of $\mathcal{AH}_i$ . 
\end{thm}
\begin{proof}
Each non-trivial bimodule over $\mathcal{AH}_i$ gives a map from
$\mathbb{Z}/2\mathbb{Z}$ to the Brauer-Picard group of $\mathcal{AH}_i$.  We
want to show that this extends to a map of categorical $2$-groups; that is, we
want to show that the obstructions $o_3$ and $o_4$ from \ref{ENOthm} vanish.  It
is enough to show that the groups that these obstructions lie in, namely
$H^3(\mathbb{Z}/2\mathbb{Z}, \operatorname{Inv}(Z(\cC)))$ and
$H^4(\mathbb{Z}/2\mathbb{Z},\CC^\times)$, vanish.  As observed above,
$\Inv(Z(\mathcal{AH}_i ) )$ is trivial, so the first obstruction group vanishes.
 Finally $H^4(\mathbb{Z}/2\mathbb{Z},\CC^\times)$ vanishes since
$\mathbb{Z}/2\mathbb{Z}$ is cyclic.   There are two categories extending each
bimodule, owing to the choice of $\alpha$ in Theorem \ref{ENOthm}.
\end{proof}

Among these new fusion categories is one with an object of dimension
$\sqrt{\frac{9+\sqrt{17}}{2}} $. The fusion graph of this small object is
(\cite{1202.4396} ): 
$$ \hpic{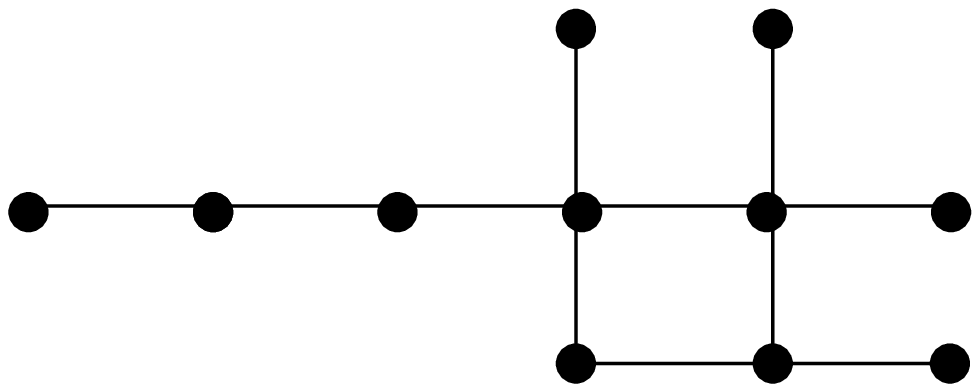} {0.6in} .$$ 

This graph is $3$-supertransitive (that is, it begins with a string of three
edges).  Outside the ADE series this is one of only two known fusion categories which is more than $2$-supertransitive.  The current record is the $4$-supertransitive fusion category $4442$ announced in \cite{1208.3637}.

\begin{rem}
It is easy to see that the Grothendieck ring of a $\mathbb{Z}/2\mathbb{Z}
$-extension of a fusion category is determined by the following data: (a) The
Grothendieck ring of the $0$-graded part; (b) the bimodule fusion rules for the
$1$-graded part considered as a bimodule category over the $0$-graded part; (c)
the dual data of the $1$-graded part. 

The full multiplication rules for each of the three non-trivial bimodule 
categories over each $\mathcal{AH}_i$, $i=1,2,3$ were given in the data
supplement to \cite{1202.4396}, available on the arxiv, so (a) and (b) are known for all of the
corresponding extensions.

Thus the only unknown information in the Grothendieck rings are the dual data of
the $1$-graded parts, which is given by an involution on the set of simple
objects in the bimodule category which preserves Frobenius-Perron dimension.
There are not very many possibilities, but we do not know how to compute these
involutions at this time. 
\end{rem}

\begin{rem}
It is natural to wonder what extensions there are of $\mathcal{AH}_i$ by the
Klein $4$-group.  This is somewhat more subtle as the obstruction group
containing $o_4$ does not vanish.
\end{rem}

\subsection{The Izumi $2^p1$ subfactors}\label{sec:Izumi-Xu}

Let $p$ be a prime, and let $R_p$ be the fusion ring generated by $g$ and $X$ with the relations $gX =
X = Xg$, $g^p = 1$, and $X^2 = \sum_i g^i + pX$ (all sums here go from $0$ to $p-1$).
 There is a left fusion module for $R_p$, which we will call $M_p^{left}$, with
basis $a, b, gb, g^2 b, \ldots , g^{p-1} b$ where $ga = a$, $X a =  \sum_i g^i
b$, and $X g^i b = a + \sum_i g^i b$.  Since $R_p$ is commutative, there is also
a right fusion module $M_p^{right}$ with the analogous fusion rules, and an
$R_p$ bimodule $M_p^{bim}$.

We will call $\cI_p$ a $\mathbb{Z}/p\mathbb{Z}$ Izumi near-group category if its
fusion ring is $R_p$.  The study of these fusion categories was initiated by
Izumi \cite{MR1832764} and has been continued recently by Evans-Gannon
\cite{1208.1500}.  

Since $g^i X \cong X$, the subcategory of $\cI_p$ generated by $g$ has $\mathrm{Vec}$ as a module
category, so this subcategory is $\mathrm{Vec}(\mathbb{Z}/p\ZZ)$.  In
particular, there is an algebra object $\mathbb{C}[\mathbb{Z}/p\ZZ]$ in $\cI_p$.

The category of left modules for this algebra is a right module category over
$\cI_p$ which we will call $\cM_p^{right}$.  It is easy to see that the fusion
rules for $\cM_p^{right}$ give the fusion module $M_p^{right}$.  Similarly we
have a left module category $\cM_p^{left}$.  Furthermore, any (left or right)
module category with fusion rules $M_p$  must be equivalent to $\cM_p$, because
the internal endomorphisms of $a$ must be $\mathbb{C}[\mathbb{Z}/p\ZZ]$.

Picking an object in a module category gives a subfactor.  The subfactor
corresponding to $g^i b$ has principal and dual-principal graphs which are
$2^p1$ spoke graphs (that is, they have $p$ spokes of length $2$ and one spoke
of length $1$).  Using \cite[\S3.3]{1202.4396} it is not difficult to see that
the resulting subfactor is independent of $i$, since the objects agree up to
tensoring with an invertible in the dual graph.  For general $p$, these $2^p1$ subfactors were first studied by Izumi.  When $p=2$ this is just the $E_6$ subfactor, and when $p=3$ there's an independent unpublished conformal inclusion construction due to Xu.  We will call this series the $2^p1$ subfactors, rather than Izumi subfactors, to avoid confusion with another series of
subfactors generalizing the Haagerup subfactor also constructed by Izumi in the
same paper (sometimes called Izumi-Haagerup subfactors).

From looking at the dual graph of the $2^p1$ subfactor, we see that the dual
category to $\cI_p$ over $\cM_p$ is another $\mathbb{Z}/p\mathbb{Z}$ near-group
category $\cI'_p$.  The fusion rules for $\cM_p$ as a $\cI_p$--$\cI'_p$ bimodule
are $M_p^{bim}$.

It is natural to wonder whether $\cI_p \cong \cI'_p$, and if so whether there's
a $\mathbb{Z}/2\mathbb{Z}$-extension of $\cI_p$ by $\cM_p$.  To that end, we
call $\cI_p$ self-dual if the dual even part is equivalent to the principal even
part.  In such a circumstance, choosing an isomorphism between the two even
parts endows $\cM_p$ with the structure of a bimodule over $\cI_p$.  In general
there may be several such choices, and there is no guarantee that any of them
have order $2$ in the Brauer-Picard group of $\cI_p$.  However, we have the
following lemma.

\begin{lma}
Suppose that $\cI_p$ is self-dual and that its outer automorphism group is
trivial.  Then the unique $\cI_p$ bimodule structure on $\cM_p$ has order $2$ in
the Brauer-Picard group.
\end{lma}
\begin{proof}
Consider $\cM_p^{-1}$ as a $\cI_p$-bimodule.  Since this has  fusion rules
$M_p$, it must be equivalent to $\cM_p$ as a left (or as a right) module
category.  Since both bimodules are invertible, this implies that $\cM_p \cong
\cM_p^{-1} \boxtimes_{\cI_p} \cF$ as bimodules where $\cF$ is an outer
automorphism.  Since the outer automorphism group is trivial, we have $\cM_p^2
\cong \id$.
\end{proof}

\begin{rem}
In general, this argument shows that for any choice of identification of $\cI_p$
with $\cI'_p$ the resulting bimodule $\cM_p$ squares to an outer automorphism.
\end{rem}

\begin{thm}\label{thm:Cp}
Suppose that $p>2$ and that $\cI_p$ is a self-dual $\mathbb{Z}/p\mathbb{Z}$
near-group category with trivial outer automorphism group.  Then there exist
exactly two $\mathbb{Z}/2\mathbb{Z}$-extensions of $\cI_p$ by $\cM_p$.
\end{thm}
\begin{proof}
We have only to show that the homomorphism,
$c:\mathbb{Z}/2\mathbb{Z}\to \operatorname{BrPic}(\cI_p),$ corresponding to $\cM_p$ is
unobstructed, i.e. that $o_3$ and $o_4$ in Theorem \ref{ENOthm} vanish.  The
obstruction $o_4$ vanishes because $H^4(G,\CC^\times)=0$ whenever $G$ is cyclic.

By Corollary \ref{app-LES}, we have the exact sequence:
$$\ast \to Hom(U(\cI_p),\CC^\times)\to \operatorname{Inv}(Z(\cI_p))\to \operatorname{Inv}(\cI_p).$$
However, the categories $\cI_p$ admit no non-trivial gradings, as is evident from
the fusion rules.  So we have that $\operatorname{Inv}(Z(\cI_p))$ includes into
$\operatorname{Inv}(\cI_p)=\mathbb{Z}/p\mathbb{Z}$.  Thus the only remaining obstruction $o_3$
lies in $H^3(\mathbb{Z}/2\mathbb{Z},\operatorname{Inv}(Z(\cI_p))$.  This obstruction group vanishes since $\operatorname{Inv}(Z(\cI))$ is
either $\mathbb{Z}/p\mathbb{Z}$ or trivial (here we use that $p$ is odd, since $H^3(\mathbb{Z}/2\mathbb{Z},\mathbb{Z}/2\mathbb{Z})$ is nontrivial).
Thus there exists a $\mathbb{Z}/2\mathbb{Z}$-extension of $\cI_p$ by $\cM_p$.

There are exactly two distinct such extensions, owing to the choice of $\alpha$
in Theorem \ref{ENOthm}.
\end{proof}

\begin{rem}
It is not difficult to work out the fusion rules for this extension of $\cI_p$. 
Since $p$ is odd, at least one of the $g^i b$ is self-dual, so without loss of
generality $b$ is self-dual.  The rest is easy to work out. 
\end{rem}

The simplest near group categories with $p$ odd are the Izumi-Xu examples where $p=3$; there
are two inequivalent such fusion categories corresponding to a choice of complex
conjugation in the structure constants for the monoidal structure. In the
following we will fix a choice of conjugation and refer to ``the'' Izumi-Xu
category for $p=3$, but everything holds equally true for either choice.

 In this case there are already two constructions of
$\mathbb{Z}/2\mathbb{Z}$-extensions, first by Ostrik in the appendix to
\cite{MR2786219}, using constructions from affine Lie algebras and conformal
embeddings, and second by Morrison-Penneys \cite{1208.3637} using planar
algebras.  In order to recover their results we need to know that the Izumi-Xu
category is self-dual and has no outer automorphisms.  Both of these facts
follow from Han's thesis \cite{han-2221}.  Self-duality follows from uniqueness
of the $2221$ subfactor up to complex conjugacy, and no outer automorphisms
follows from the explicit quadratic relations satisfied by Han's generators (as
in \cite[Lemma 5.3]{MR2909758} and \cite[Thm. 4.9]{1202.4396}).

\begin{rem}
In fact, the Brauer-Picard $1$-groupoid of the Izumi-Xu fusion category $\cI_3$
is a single point with automorphism group $\mathbb{Z}/2\mathbb{Z}$.  Following
the approach in \cite{MR2909758}, the only possible minimal algebra objects in
the Izumi-Xu fusion category are $1$ and $1+g+g^2$, and those each have unique
algebra structures.  Thus, the only simple module categories are $\cI_3$ and
$\cM_3$.  Since $\cI_3$ is self-dual and has no outer automorphisms, the only
nontrivial $\cI_3$--$\cD$ bimodule is $\cM_3$ where $\cD \cong \cI_3$.
\end{rem}

We are optimistic that the other near group categories coming from Izumi and
Evans-Gannon also give $\mathbb{Z}/2\mathbb{Z}$-extensions.  In theory, it
should be possible to work out whether these categories are self-dual and what
their outer automorphism groups are from the detailed descriptions given by
Izumi and Evans-Gannon, but in practice this may be somewhat difficult to
extract.

\section{The homotopy fiber sequence} \label{sec4}

We begin by recalling the definitions of tensor functor and natural
transformation of tensor functors, primarily to fix notation.

A tensor functor $(F,J_F): \cC\to \cD$ is a functor $F$ of abelian categories,
together with a natural isomorphism,
$$J_F: F\circ\ot_\cC \xrightarrow{\sim} \ot_\cD\circ F\bt F,$$
satisfying a certain cocycle condition\footnote{To ease notation, we adopt the usual convention of refering to tuples $(F,\cdots)$, consisting of a functor with structural isomorphisms (such as tensor or module functor structure) simply by ``$F$".}.  More precisely, $J_F$ consists of a
family of isomorphisms,
$$J_F: F(X\ot Y)\xrightarrow{\sim} F(X)\ot F(Y),$$
natural in $X,Y\in \cC$, and such that the following diagram commutes:

$$\xymatrix{ F((X\ot Y) \ot Z) \ar[r]^{F ( \alpha_{\cC})} \ar[d]_{J_F} & F(X\ot
(Y \ot Z))\ar[d]^{J_F}\\ F(X \ot Y)\ot F(Z) \ar[d]_{J_F} & F(X)\ot F(Y \ot Z)
\ar[d]^{J_F} \\ (F(X) \ot F(Y)) \ot F(Z) \ar[r]^{\alpha_{\cD}} & F(X) \ot (F(Y)
\ot F(Z) )}$$

where $\alpha_{\cC} $ and $\alpha_{\cD} $ are the associators of $\cC $ and $\cD
$ respectively.  In particular, we note that the tensor structure $J_F$ is
additional data packaged with $F$.

A tensor functor $(F,J_F)$ is an equivalence if $F$ is an equivalence of abelian
categories.
\begin{mydef}A natural transformation $\theta: F\to G$ between tensor functors
$F,G:\cC\to\cD$ is monoidal if for all $X,Y\in \cC$, the following diagram
commutes:
$$\xymatrix{ F(X\ot Y) \ar[r]^{\theta_{X\ot Y}} \ar[d]_{J_F} & G(X\ot
Y)\ar[d]^{J_G}\\ F(X)\ot F(Y)\ar[r]^{\theta_X\ot\theta_Y}& G(X)\ot G(Y)}$$
\end{mydef}

Morphisms between module categories and bimodule categories are defined in an
analagous way; namely, they are functors of abelian categories along with certain
natural transformations making certain diagrams commute. An equivalence of
module categories is a module functor whose underlying functor is an equivalence
of abelian categories.  See \cite{MR1976459} for details.

We now construct the functors $F_-$ and $M_-$ from Section
\ref{sec:MainResults}.  For clarity of exposition, we suppress associators for
the remainder of this section.
\subsection{Construction of $F_-$}
Given $g\in\Inv(\cC)$, we define $F_g\in\Eq(\cC)$ as follows.  For $X, Y\in
\cC$, and $\rho:X\to Y$, we set:
$$F_g(X):=g\ot X\ot g^{-1},\quad F_g(\rho):=\id_g\ot\rho\ot\id_{g^{-1}}.$$
We equip $F_g$ with the tensor structure:{\small
$$F_g(X)\ot F_g(Y) = g\ot X\ot g^{-1}\ot g \ot Y\ot g^{-1}
\xrightarrow[ev_g]{\sim} g\ot X\ot Y\ot g^{-1}=F_g(X\ot Y).$$}
The assignment $g\mapsto F_g$, extends to a functor, 
$$F_-:\Inv(\cC)\to\Eq(\cC),$$ by assigning to any $\phi:g\xrightarrow{\sim} h$
the natural isomorphism,
$$F_g = g\ot -\ot g^{-1} \xrightarrow[\phi\ot\id\ot(\phi^{-1})^*]{\sim} h\ot
-\ot h^{-1} = F_h.$$

The obvious natural isomorphisms, $F_{g\otimes h}\cong F_g\circ F_h$, endow
$F_-$ with a monoidal structure.

\subsection{Construction of $M_-$}
For every $F\in \Eq(\cC)$, we define an invertible $\cC$-$\cC$-bimodule category
$M=M_F$, by letting $M=\cC$ as an abelian category, and defining, for
$X,Y\in\cC$, $m\in M$:
$$X\ot_{M_F} m \ot_{M_F} Y := X\ot m\ot F(Y).$$
Here, unadorned tensor products denote the tensor product in $\cC$.
For every natural isomorphism $\theta: F\to G$, we have a $\cC$-$\cC$-bimodule
auto-equivalence $(\id_\cC,J_\theta):M_F\to M_G$, which is the identity functor
of $\cC$, and where $J_\theta:F(Y)\to G(Y)$ is applied before tensoring on the
right.

This gives a functor,
$$M_-:\Eq(\cC)\to \Out(\cC),$$
$$ F\mapsto M_F$$
of $2$-categories, where we regard $\Eq(\cC)$ as a categorical $2$-group with
strict associativity of $1$-morphisms.

We have a bi-additive functor of $\cC$-$\cC$-bimodules,
$$\hat{\ot}:M_F\bt M_G\to M_{F\circ G}$$
$$ m\bt n \mapsto m \ot F(n),$$
and isomorhpisms,
$$M\hat{\ot} (A\ot N) = M \ot \phi(A\ot N) \xrightarrow{J_F} M\ot\phi(A)\ot
\phi(N) = (M\ot A)\hat{\ot} N.$$
natural in $M,A$ and $N$.  Thus $\hat{\ot}$ defines a functor,
$$\iota_{F,G}: M_F \ot_\cC M_G\to M_{F\circ G},$$
of abelian categories, which is clearly an equivalence.  Moreover, we have an
isomorphism of functors,
$$J_{F,G}: \ot_{M_{F\circ G}}\circ\iota_{F,G} \to
\iota_{F,G}\circ\ot_{M_F\ot_\cC M_G},$$
$$ X\ot m\ot F(n)\ot F(G(Y)) \xrightarrow{J_F^{-1}} X\ot m\ot F(n\ot G(Y)).$$
Thus, $M_-$ induces a homomorphism of categorical $2$-groups, which we also
denote $M_-$.

\subsection{The homotopy fiber of $M$}

Let $p:\mathcal{G}\to\mathcal{H}$ be a homomorphism of categorical $n$-groups. 
The homotopy fiber, $p^{-1}(X)$, of $X\in\mathcal{H}$ has as its objects pairs
$(Y\in\mathcal{G}, \phi:X\xrightarrow{\sim} p(Y))$.  Morphisms are those
inherited from $\mathcal{H}$; that is:
$$\Hom_{p^{-1}(X)}((Y_1,\phi_1),(Y_2,\phi_2) :=
\Hom_\mathcal{H}(\phi_1,\phi_2).$$
Because all objects and morphisms in $\mathcal{H}$ are invertible, Quillen's
Theorem B \cite{MR0338129} asserts $p^{-1}(X)\to \mathcal{G}\to\mathcal{H}$ is a
homotopy fiber sequence, for any object $X\in \mathcal{H}$.

In this section, we construct an equivalence between the full subcategory
$\Inv(\cC)$ of invertible objects in $\cC$, and the homotopy fiber $M^{-1}(\cC)$
over the trivial $\cC$-$\cC$-bimodule $\cC$.

We have equivalences of bimodule categories,
$$\Phi_g:\cC\to M_{F_g}$$
$$ X\mapsto X\ot g^{-1},$$
equipped with tensor structure,{\small
$$X\ot \Phi_g(Y) \ot F_g(Z) = X\ot Y\ot g^{-1}\ot g\ot Z \ot g^{-1}
\xrightarrow[ev_g]{\sim} X\ot Y\ot Z\ot g^{-1} =\Phi_g(X\ot Y\ot Z).$$}

We thus construct a functor:

$$M_{F_{-}}:\Inv(\cC)\to M^{-1}(\cC).$$
$$M_{F_{-}}(g):= (M_{F_g},\Phi_g).$$
\begin{thm} The functor $M_{F_-}$ is an equivalence.\label{fib-thm}\end{thm}
\begin{proof}
Let $(F,J_F)\in\Eq(\cC)$, and suppose we have an equivalence,
$$(\phi,J_\phi):\cC\to M_F,$$ of bimodule categories.  Clearly, $\phi$
induces an auto-equivalence of $\cC$ as a left-module category (recall that
$M_F=\cC$ canonically, as a left $\cC$-module category). It is then routine to
see that $\phi= -\ot g^{-1}$, where $g^{-1}=\phi(\mathbf{1})\in\Inv(\cC)$.

The tensor data $J_\phi$ therefore consists of a family of isomorphisms,
$$X\ot Y \ot Z \ot g = \phi(X\ot Y\ot Z) \xrightarrow[J_\phi]{\sim}  X\ot
\phi(Y)\ot F(Z) = X\ot Y \ot g \ot F(Z),$$
which is natural in $X, Y, Z \in \cC$.  In particular, taking $X,Y =
\mathbf{1}$, we obtain a natural isomorphism,
$$F \cong g\ot - \ot g^{-1}=F_g.$$
We thus have a canonical isomorphism $(M_F,\phi)\cong (M_{F_g},\Phi_g)$ in the
homotopy fiber, so that the induced functor $F_-:\Inv(\cC)\to M_\cC$ is
essentially surjective.  It remains to show that it is fully faithful.  We need
to show that every equivalence, $\theta: M_{F_g}\xrightarrow{\sim} M_{F_h}$, of
bimodule categories making the diagram:
$$\xymatrix{M_{F_g}\ar[dr]_{\Phi_g^{-1}}\ar[rr]^{\theta}&&
M_{F_h}\ar[dl]^{\Phi_h^{-1}}\\&\cC&}$$
commutative is in fact induced by a unique isomorphism $g\cong h$.  The
isomorphism required for fullness is
$g=F_g(\mathbf{1})\xrightarrow{\theta_\mathbf{1}} F_h(\mathbf{1})=h$.  The
bimodule compatibility condition for $\theta$ implies that all the $\theta_X$
are determined by $\theta_{\mathbf{1}}$, thus implying faithfulness.
\end{proof}

Theorem \ref{fib-thm} now implies \ref{homotopy-thm}, as it allows us to
identify $\Inv(\cC)$ with a homotopy fiber of the functor $M_-$.  Corollary
\ref{maincor} follows from the long exact sequence in homotopy groups for a
fibration.

\begin{proof}[Proof of Corollary \ref{app-LES}]
All that remains is to compute the connecting homomorphisms
$\delta_2:\pi_2\Out(\cC)\to\pi_1\Inv(\cC)$, and
$\delta_1:\pi_1\Out(\cC)\to\pi_0\Inv(\cC).$

To this end, we identify $\pi_1\Out(\cC)$ with the automorphisms of the unit
object of $\Out(\cC)$, namely the $\cC$-$\cC$-bimodule auto-equivalences of the
regular bimodule $\cC$.  Every such auto-equivalence, $\phi$, is isomorphic to
the functor of tensoring by a central object $g$.  Thus in $\Eq(\cC)$, we have
an isomorphism $\phi \cong -\ot g$, and thus the connecting homomorphism
$\delta_1:\pi_1\Out(\cC)\to\pi_0\Inv(\cC)$ is simply the forgetful functor
$For:\Inv(Z(\cC))\to \Inv(\cC)$.  Similarly, $\delta_2$ is the induced homomorphism
$For:Aut_{Z(\cC)}(\mathbf{1},\mathbf{1})\to Aut(\mathbf{1},\mathbf{1})$.  In particular, $\delta_2$ is an isomorphism, so that
the first two terms of the sequence split off, yielding Corollary \ref{app-LES}.
\end{proof}

\newcommand{\urlprefix}{}
\bibliographystyle{alpha}
\bibliography{bibliography}
\end{document}